\documentclass[11pt]{article}
\usepackage{amssymb,amsmath,amsthm}

\usepackage{bbm}
\usepackage{color}

\title{Orlicz-space Hardy and Landau-Kolmogorov\\ inequalities
for Gaussian measures}
\author{Krzysztof Oleszkiewicz\thanks{Supported in part by MNiSW grant
N N201 397437.}\,\,\,\,and Katarzyna Pietruska-Pa\l{}uba\thanks{Supported in part
  by  KBN grant no.
                         1-PO3A-008-29. }\\
                         {\small(Warsaw)}}

\newtheorem{theo}{\bf Theorem}[section]
\newtheorem{cor}{\bf Corollary}[section]
\newtheorem{lem}{\bf Lemma}[section]
\newtheorem{rem}{\bf Remark}[section]

\newtheorem{prop}{\bf Proposition}[section]

\makeatletter \@addtoreset{equation}{section}

\makeatother

\newcommand{\ds}{\displaystyle}

\newcommand{\ts}{\textstyle}

plus 6pt minus 12pt
\newcommand{\barint}{
         \rule[.036in]{.12in}{.009in}\kern-.16in
          \displaystyle\int  }

\def\rn{{\mathbb{R}^{n}}}

\def\M{{\cal M}}
\def\zi{[0,\infty )}

\def\C{{{C}}}
\def\K{{{\cal{K}}}}
\def\L{{{\cal{L}}}}
\def\M{{{\cal{M}}}}

\def\R{{\mathbb R}}

\def\eps{{\varepsilon}}
\def\1{{\mathbf 1}}
\def\ud{{\mathrm d}}
\def\e{{\mathrm e}}

\begin{document}

\date{} \maketitle

\begin{abstract}
We prove Orlicz-space versions of Hardy and Landau-Kolmogorov
inequalities for Gaussian measures on $\rn.$
\end{abstract}

\section{Introduction}

The classical Hardy inequality on $\rn$ states that for $u\in
W^{1,2}(\rn)$
\begin{equation}
\label{hardyclas} \int_\rn \frac{u^2(x)}{|x|^2}\,\ud x \leq
\left(\frac{2}{N-2}\right)^2\int_{\rn}|\nabla u|^2\,\ud x,
\end{equation}
which can be written as
\[\left\|u(\cdot)\frac{1}{|\cdot|}\right\|_{2}\leq\frac{2}{N-2}\|\nabla
u\|_2.\]

It is a natural question to ask for its generalisations: the
`measure' $\frac{1}{|x|^2}\,\ud x$ on the left hand side of
(\ref{hardyclas}) can be replaced by $\ud\mu,$ second norm by $p-$th
or $q-$th, the measure $\ud x$ on the right hand side by $\ud\nu.$

For $n=1$ and functions $u$ vanishing on the boundary, the  Hardy
inequality (for general measures on $[a,\infty)$) in $L^p$ norms
has been thoroughly studied and there is a complete
description of measures that allow for such an
inequality. We have the following characterization, which can be found in
(\cite{ma}, Section 1.3.1, Th. 1):

\begin{theo}[\cite{ma}] \label{mazzz}
Suppose that $\mu,$ $\nu$ are
nonnegative measures on $(a,\infty),$ let $\nu^*$ be the
absolutely continuous part of $\nu.$ Then  the inequality
\begin{equation}\label{mazhar}
\left(\int_a^\infty \left|\int_a^x f(t)\,\ud t\right|^q
\,\ud\mu(x)\right)^{\frac{1}{q}}\leq
 C \left(\int_a^\infty |f(x)|^p\,\ud\nu(x)\right)^{\frac{1}{p}},
 \end{equation}
 where $1\leq p\leq q \leq \infty,$ holds for all Borel measurable
 functions $f$ if and only if
 \begin{equation}\label{mazjacond}
B=\sup_{r>a}\,[\mu([r,\infty))]^{\frac{1}{q}}\left(\int_a^r
\left(\frac{\ud\nu^*}{\ud x}\right)^{\frac{-1}{p-1}}\right)^{\frac{p-1}{p}}<\infty.
 \end{equation}
\end{theo}

We are concerned with generalisations of (\ref{hardyclas}), when
the Lebesgue measure is replaced by the standard Gaussian measure on
$\rn,$ $\gamma_n(\ud x)=
(2\pi)^{-n/2}\exp(-|x|^{2}/2)\,\ud x,$  the `inner'
weight $w(x)=|x|^{-1}$ is replaced by $w(x)=|x|,$ and the
$L^p-$norms are replaced by Orlicz norms or Orlicz modular
expressions. Inequalities for the Gaussian measure  on $\rn$ can be reduced
to inequalities on $\zi,$ with respect to the measure ${\rm d}\mu_n(r)= r^{n-1}e^{-r^2/2}\ud r.$
Applying Theorem \ref{mazzz} with $\ud \mu(r)= r^q \ud\mu_n(r)$ and $\ud\nu(r)=\ud\mu_n(r),$
 we see that in that case inequality (\ref{mazhar}) with $p=q$
 (for Hardy transforms) can hold
only if $p>n.$

Another reason why inequalities (\ref{mazhar}) need to be extended
is that we need
inequalities for measures  $\mu_n$  holding true
 not only for Hardy transforms,
but also for functions not necessarily vanishing at zero.

More precisely, in this paper we aim at obtaining inequalities of the form:
\begin{equation}\label{what0}
\int_0^\infty M(|rv(r)|)\,\ud\mu_n(r) \leq C_1 \int_0^\infty
M(|v(r)|)\,\ud\mu_n(r)+C_2 \int_0^\infty M(|\nabla v(r)|)\,\ud\mu_n(r),
\end{equation}
which then give rise to
\begin{equation}\label{what}
\int_\rn M(|xv(x)|)\,\ud\gamma_n(x) \leq C_1 \int_\rn
M(|v(x)|)\,\ud\gamma_n(x)+C_2 \int_\rn M(|\nabla v(x)|)\,\ud\gamma_n(x),
\end{equation}
with $C_1, C_2$ independent of $v$ from a sufficiently large class
of functions, but  depending on the dimension $n.$ This
dependence cannot be suppressed,  see the discussion at the end of Section \ref{hardygauss}.
We still  call  inequality  (\ref{what}) the Hardy
inequality for the Gaussian measure. The Hardy inequality for Gaussian measures are of separate interest,
both in the probability theory and the PDE theory. For other
inequalities for the Gaussian measure (Poincar\'e, log-Sobolev) the
reader can consult e.g. \cite{led}, while in \cite{bla,bfp, bct}
one can find the results concerning the  importance of Gaussian
measures in the PDE theory.

We obtain (\ref{what0}) and (\ref{what}) for general $N-$functions $M$ satisfying the $\Delta_2-$condition (doubling), see Proposition
\ref{wklep1}.
With some additional condition (close to the property that $M(r)/r^2$ is non-decreasing) we were able
to provide a more detailed analysis of the resulting constants.

Inequalities (\ref{what})
are an example of the so-called $U-$bounds (see \cite{Heb-Zeg}),
i.e. inequalities of the form
\[ \int |v|^q  U \ud\mu \leq C\int |\nabla f|^q \ud\mu + D\int |f|^q \,\ud \mu \]
analyzed in the context of general metric spaces and metric gradients.
Examples furnished in that paper indicate that
 the most interesting $U-$bounds for a measure $\ud \mu(x)={\rm e}^{-\varphi(x)}$
 are those with $U(x)=|\nabla \varphi|.$ Such inequalities can be related e.g. to Poincar\'{e}, log-Sobolev
 and other
 inequalities for $\mu.$
 Since for the Gaussian measure one has $\varphi (x)=|x|^2/2,$ and $|\nabla \varphi (x)|=|x|,$
the  weight $w(x)=|x|$ in (\ref{what}) is the most desirable one.
 In a somewhat different context, such inequalities were also investigated in \cite{bct}.

As an application we show, using a general
theorem from \cite{AKKPPbullpan},
 that   inequality  (\ref{what})  implies the Orlicz version of the Landau-Kolmogorov
inequality for the Gaussian measure:
\begin{equation}\label{star}
\|\nabla u\|_{L^M(\rn,\gamma_n)}\leq C_1
\|u\|_{L^{M}(\rn,\gamma_n)}+C_2\|\nabla^{(2)}u\|_{L^M(\rn,\gamma_n)},
\end{equation}
together with its modular counterpart.

In \cite{AKKPPBBMS}, one proves additive Gagliardo-Nirenberg
inequalities in weighted Orlicz spaces. In particular,    the
following  inequality for Gaussian measures was obtained:
\[
\|\nabla u\|_{L^M(\rn,\gamma_n)}\leq C_1
\|u\|_{L^{\Phi_1}(\rn,\gamma_n)}+C_2\|\nabla^{(2)}u\|_{L^{\Phi_2}(\rn,\gamma_n)},
\]
where $M$ was an $N-$function satisfying the
$\Delta_2-$condition and increasing faster that $r^2$, and
$\Phi_1,\Phi_2$ were other $N-$functions. The functions $M,
\Phi_1, \Phi_2$ were tied by certain consistency conditions, which
in particular {\em excluded} the case $M=\Phi_1=\Phi_2$, i.e. the
results of \cite{AKKPPBBMS} did not yield the Landau-Kolmogorov
inequality (\ref{star}) in Orlicz norms.
This is rectified in present paper, see Corollary \ref{forgausssian}.

 \section{Preliminaries}
\subsection{Notation}

 Throughout the paper, the symbol
$\nabla^{(2)}u$ denotes the Hessian of a function $u\in C^2(\rn),$
i.e. the matrix $[\frac{\partial^2u}{\partial x_i\partial
x_j}]_{i,j=1}^n.$ For a square $n\times n$ matrix $A$, by $|A|$
 we denote its Hilbert-Schmidt
norm:
\[|A|=|A|_{HS}=\left(\sum_{i,j=1}^n a_{ij}^2\right)^{\frac{1}{2}};\]
 $C_0^\infty(\rn )$ stands for
smooth compactly supported functions
 on $\rn.$

\subsection{$N-$functions}

A function $M\!:\![0,\infty )\to [0,\infty )$ is called an
{$N-$function} if it is convex, $M(0)=0,$ $\ds\lim_{r\to
0^+}{M(r)}/{r} =0$ and $\ds\lim_{r\to\infty}{M(r)}/{r}  =\infty.$
 An $N-$function $M$ is said to
satisfy the $\Delta_2-$condition if and only if
\begin{equation}\label{d3}
 \exists\, C_{M}>1 \; \forall\, r>0 \;\;\;  M(2r)\leq C_M  M(r).
\end{equation}

If additionally $M$ is differentiable, then the
$\Delta_2-$condition (\ref{d3}) is equivalent to the existence of
$D_M>1$ s.t.  $M'(r)\leq D_M \frac{M(r)}{r},$  for $r>0.$
Additional conditions on $M$ will be added as needed.

\subsection {Weighted Orlicz spaces}

 Suppose that $\mu$ is a positive Radon
measure on $\rn$
 and let $M:[0,\infty)\to[0,\infty)$ be an $N-$function.
 The weighted space $L^{M}(\mu)$ with respect  to the measure $\mu$ is, by definition, the function space
\[L^M(\rn,\mu)=L^{M}(\mu)\stackrel{def}{=}\left\{f\mbox{ measurable} \colon   \int_{\rn}
 M\left(\frac{|f(x)|}{K}\right)\,\ud\mu(x)\leq 1\ \mbox{ for some }\ K >0\right\} ,  \]
 equipped with the Luxemburg norm
\[\|f\|_{L^M(\mu)}=\inf\left\{ K>0 \colon \int_{\rn}M\left(\frac{|f(x)|}{K}\right)\,\ud\mu(x)\leq
1\right\} .\] This norm is complete and turns $L^{M}(\mu)$ into a Banach
space. For $M(r)=r^p$ with $p> 1$,  the space $L^{M}(\mu)$
coincides with the usual $L^p(\mu)$ space.

 We recall the following two properties of Young
functionals: for every $f\in L^{M}(\mu)$ we have
 \begin{equation}\label{norm}
 \|f\|_{L^M(\mu)}\leq \int_{\rn}M(|f(x)|)\,\ud\mu(x)+1,
 \end{equation}
and
 \begin{equation}\label{norm1}
 \int_{\rn} M\left(\ts\frac{|f(x)|}{\|f\|_{L^M(\mu)}}\right)\,\ud\mu(x)\leq
 1.
 \end{equation}
 When $M$ satisfies the $\Delta_2-$condition, then
 (\ref{norm1}) becomes an equality.\\

For more information on Orlicz spaces the reader may consult e.g. \cite{rao-ren}.

\section{The Hardy inequality for the Gaussian measure}\label{hardygauss}

\subsection{Inequalities on the real line}

We start with  inequalities for measures
 $\mu_n(\ud{ r})=r^{n-1} {\rm e}^{-r^2/2}{\rm d}r,$ $r>0,$
  where $n=1,2,...$
In our approach, we will
make the following assumption concerning the function $M:$
\begin{description}
 \item[{\bf(M)}] $M:[0,\infty) \rightarrow [0,\infty)$  is
 nonconstant and there exist $d_M, D_M>0$
such that $M$ satisfies the inequalities
\begin{equation} \label{gw}
\forall_{r \geq 0, a \geq 1}\,\,\,\, M(ar) \leq a^{D_{M}} \cdot M(r)
\end{equation}
and
\begin{equation} \label{gwgw}
\forall_{r \geq 0, a \in (0,1)}\,\,\,\, M(ar) \leq a^{d_{M}} \cdot M(r).
\end{equation}
\end{description}

Then, obviously, $D_{M} \geq d_{M}$ and $M$ is an increasing
continuous function with $M(0)=0$, $\lim_{r \to \infty}
M(r)=\infty$, and moreover $r\mapsto r^{-d_M}M(r)$ is
non-decreasing.

When we additionally assume that $D_{M}>2$ and $d_{M} \geq 2,$
then
 in particular\linebreak
  $\lim_{r \to 0^{+}} r^{-2}M(r)$ exists and is
finite. Hence by a natural convention we treat $r \mapsto
r^{-2}M(r)$ and $r \mapsto r^{-1}M(r)$ as continuous functions on
the whole $[0,\infty)$, the latter taking value $0$ at $0$.

\begin{lem}
Assume that $M$ satisfies {\bf (M)} with $d_M\geq 2,$ $D_M>2.$
Then for any $\lambda \geq 1/d_{M}$ and $r,s \geq 0$ we have
\begin{equation} \label{ni1}
r^{-1}M(r) \cdot s \leq (1-D_{M}^{-1})(\lambda D_{M})^{-1/(D_{M}-1)}M(r)+\lambda M(s)
\end{equation}
and
\begin{equation} \label{ni2}
r^{-2}M(r) \cdot s^{2} \leq (1-2D_{M}^{-1})(\lambda D_{M})^{-2/(D_{M}-2)}M(r)+2\lambda M(s).
\end{equation}
\end{lem}

\begin{proof}
Because of the continuity we may and will assume that $r$ and $s$ are strictly positive. Let $\alpha \in \{1,2\}$ and
$\psi_{\alpha}(u)=u^{\alpha}-\alpha \lambda u^{D_{M}}$. Since
\[
(1-\alpha D_{M}^{-1})(\lambda D_{M})^{-\alpha/(D_{M}-\alpha)}=\sup_{u \in [0,1]} \psi_{\alpha}(u),
\]
by setting $u=s/r$ we rewrite both asserted inequalities in the case $s \leq r$ as
\[
\psi_{\alpha}(u) \leq (1-\alpha D_{M}^{-1})(\lambda D_{M})^{-\alpha/(D_{M}-\alpha)}+\alpha \lambda
\left( \frac{M(s)}{M(r)}-(s/r)^{D_{M}}\right),
\]
so that they immediately follow from $M(r)=M(rs^{-1} \cdot s)\leq (r/s)^{D_{M}}M(s)$.

For $s \geq r$ we have $M(r)=M(rs^{-1} \cdot s)\leq (r/s)^{d_{M}}M(s)$, so by setting $u=s/r$ we reduce our task to proving
\[
\forall_{u \geq 1}\,\,\,\, u^{\alpha} \leq (1-\alpha D_{M}^{-1})(\lambda D_{M})^{-\alpha/(D_{M}-\alpha)}+
\alpha \lambda u^{d_{M}}.
\]
The case $u=1$ of the above estimate follows by the previous argument, and the proof is finished by observing that
$\frac{\ud}{\ud u}u^{\alpha} \leq \alpha \lambda \cdot \frac{\ud}{\ud u}u^{d_{M}}$ for $u>1$ because $d_{M} \geq \alpha$
and $\lambda \geq 1/d_{M}$.
\end{proof}

\begin{prop} \label{alternatywa}
Assume that a non-constant function $M:[0,\infty) \rightarrow
[0,\infty)$ satisfies {\bf (M)} with $d_{M} \geq 2$ and
$D_{M}>2$. Let $n \geq 1$ and $\ud
\mu_{n}(r)=r^{n-1}\e^{-r^{2}/2}\,\ud r$. For a continuous and
piecewise $\C^{1}-$function $u:[0,\infty) \rightarrow \R$ set
\[
\!\!\!\K=\int_{0}^{\infty}\! M(r|u(r)|)\,\ud \mu_{n}(r),
\L=\int_{0}^{\infty}\! M(|u(r)|)\,\ud \mu_{n}(r),
\M=\int_{0}^{\infty}\! M(|u'(r)|)\,\ud \mu_{n}(r).
\]
Then
\begin{eqnarray} \label{term1}
\K &\leq& (D_{M}/d_{M})^{D_{M}/(D_{M}-2)}\L \;\; \leq \;\;
\e^{D_{M}/2}
\cdot \L\\[2mm]
&\mbox{\it or} &\nonumber\\
 \label{term2} \K &\leq& \left( \frac{1}{2}D_{M}\M^{1/D_{M}}+
\sqrt{\frac{1}{4}D_{M}^{2}\M^{2/D_{M}}+(D_{M}+n-2)\L^{2/D_{M}}}\right)^{D_{M}}.
\end{eqnarray}
If additionally $D_{M}+n \geq \e+2$ (which holds true whenever $n
\geq 3$), then
\[
\K \leq
\left( \frac{1}{2}D_{M}\M^{1/D_{M}}+\sqrt{\frac{1}{4}D_{M}^{2}\M^{2/D_{M}}+(D_{M}+n-2)\L^{2/D_{M}}} \right)^{D_{M}}.
\]
\end{prop}

\begin{proof}
We have $D_{M}/d_{M} \leq D_{M}/2 \leq \e^{\frac{D_{M}}{2}-1}$,
which proves the second inequality of (\ref{term1}). Therefore, if
$D_{M}+n-2 \geq \e$ then the right-hand side of (\ref{term2})
dominates the right-hand side of (\ref{term1}), which proves the
last assertion. Hence it suffices to prove that (\ref{term1}) or
(\ref{term2}) holds true. Additionally, let us assume at first
that $u$ is compactly supported. By a standard integration by
parts argument we obtain
\begin{eqnarray*}
\K&=&-\int_{0}^{\infty} M(r|u(r)|)r^{n-2}\frac{\ud}{\ud
r}(\e^{-r^{2}/2})\,\ud r \\
&\leq &D_{M}\int_{0}^{\infty} \frac{M(r|u(r)|)}{r|u(r)|} \cdot
|u'(r)|\,\ud \mu_{n}(r)+ (D_{M}+n-2)\int_{0}^{\infty}
\frac{M(r|u(r)|)}{(r|u(r)|)^{2}} \cdot u(r)^{2}\,\ud \mu_{n}(r).
\end{eqnarray*}
We have used the fact that $\frac{M(r|u(r)|)}{r|u(r)|}|u(r)|r^{n-1}\e^{-r^{2}/2} \Big|_{0}^{\infty}=0$ since $M(x)/x=0$ for $x=0$,
and the fact that (\ref{gw}) implies $\limsup_{y \to x} \frac{M(y)-M(x)}{y-x} \leq D_{M} \frac{M(x)}{x}$.
Now, for any $\lambda, \rho \geq 1/d_{M}$ we can apply (\ref{ni1}) to estimate the first summand, and (\ref{ni2}) to bound
the second summand, arriving at
\begin{eqnarray}
\K &\leq& D_{M}(1-D_{M}^{-1})(\rho
D_{M})^{-1/(D_{M}-1)}\K+D_{M}\rho\M \nonumber\\
&& +   (D_{M}+n-2)(1-2D_{M}^{-1})(\lambda
D_{M})^{-2/(D_{M}-2)}\K \nonumber\\
&& +2(D_{M}+n-2)\lambda\L. \label{ni3}
\end{eqnarray}
If $\K \leq (D_{M}/d_{M})^{D_{M}/(D_{M}-2)}\L$ then our assertion is trivially satisfied.\\
If  $\K \leq (D_{M}/d_{M})^{D_{M}/(D_{M}-1)}\M$ then
$\K \leq (D_{M}/2)^{D_{M}/(D_{M}-1)}\M \leq (D_{M}/2)^{D_{M}}\M$, and again there is nothing to prove.
Hence we may and do assume that
\[
\K \geq \max \left( (D_{M}/d_{M})^{D_{M}/(D_{M}-2)}\L, (D_{M}/d_{M})^{D_{M}/(D_{M}-1)}\M \right),
\]
so that $\lambda_{0}=D_{M}^{-1}(\K/\L)^{(D_{M}-2)/D_{M}}$ and $\rho_{0}=D_{M}^{-1}(\K/\M)^{(D_{M}-1)/D_{M}}$
satisfy $\lambda_{0}, \rho_{0} \geq 1/d_{M}$. By setting $\lambda=\lambda_{0}$ and $\rho=\rho_{0}$ in (\ref{ni3}) we obtain
\[
\K \leq D_{M}\M^{1/D_{M}}\K^{(D_{M}-1)/D_{M}}+(D_{M}+n-2)\L^{2/D_{M}}\K^{(D_{M}-2)/D_{M}},
\]
so that
\[
\left( \K^{1/D_{M}}-\frac{1}{2}D_{M}\M^{1/D_{M}}\right)^{2} \leq \frac{1}{4}D_{M}^{2}\M^{2/D_{M}}+(D_{M}+n-2)\L^{2/D_{M}},
\]
which ends the proof in the case of compactly supported $u$.

In the general case let $N \geq 1$, $u_{N}(r)=u(r)$ if $r \in [0, N]$, $u_{N}(r)=\frac{2N-r}{N}u(r)$ if $r \in [N, 2N]$,
and $u_{N}(r)=0$ if $r \geq 2N$. Let
\[
\K_{N}=\int_{0}^{\infty} M(r|u_{N}(r)|)\,\ud \mu_{n}(r),\;\;\;
\L_{N}=\int_{0}^{\infty} M(|u_{N}(r)|)\,\ud \mu_{n}(r),
\]
\[
\M_{N}=\int_{0}^{\infty} M(|u_{N}'(r)|)\,\ud \mu_{n}(r).
\]
Since $u_{N}$ is compactly supported we have
\[
\K_{N} \leq (D_{M}/d_{M})^{D_{M}/(D_{M}-2)}\cdot \L_{N}
\]
or
\[
\K_{N} \leq \left(\frac{1}{2}D_{M}\M_{N}^{1/D_{M}}+
\sqrt{\frac{1}{4}D_{M}^{2}\M_{N}^{2/D_{M}}+(D_{M}+n-2)\L_{N}^{2/D_{M}}}\right)^{D_{M}}.
\]
By the Monotone Convergence Theorem we obviously have $\K_{N} \to \K$ and $\L_{N} \to \L$
as $N \to \infty$ (note that $|u_{N}| \nearrow |u|$ and recall that $M$ is non-decreasing). Since there is
$u_{N}'(r)=\frac{2N-r}{N}u'(r)-\frac{1}{N}u(r)$ for all points $r \in (N, 2N)$ at which $u$ is differentiable, we get
$|u_{N}'(r)| \leq |u'(r)|+\frac{1}{N}|u(r)|$ for almost all $r>0$. Let
$
A_{N}=\{ r>0: u'(r) \hbox{\,\,exists and\,\,} |u(r)| \leq N^{1/2}|u'(r)|\}
$
and let $B_{N}=(0,\infty) \setminus A_{N}$. If $r \in A_{N}$ then
\[
M(|u_{N}'(r)|) \leq M\left((1+N^{-1/2})|u'(r)|\right) \leq \left(1+N^{-1/2}\right)^{D_{M}}M(|u'(r)|)
\]
whereas for almost all $r \in B_{N}$ we have
\[
M(|u_{N}'(r)|) \leq M\left((N^{-1}+N^{-1/2})|u(r)|\right)\leq \left(2/\sqrt{N}\right)^{D_{M}}M(|u(r)|).
\]
Hence $\M_{N} \leq \left(1+N^{-1/2}\right)^{D_{M}}\M+\left(2/\sqrt{N}\right)^{D_{M}}\L \stackrel{N \to \infty}{\longrightarrow} \M$, and the proof is finished.
The argument fails only
if $\L=\infty$, but then the main assertion is trivial.
\end{proof}

We may  slightly weaken the assertion of Proposition \ref{alternatywa} by turning it into a more convenient linear estimate:
\begin{equation} \label{liniowe}
\K \leq C_{1}\L+C_{2}\M,
\end{equation}
with positive $C_{1}$ and $C_{2}$ depending only on $n$, $D_{M}$
and $d_{M}$.
Elementary calculations permit us to obtain e.g.
\begin{eqnarray*}\label{stale-np}
C_1=2^{D_M-1}(D_M+n-2)^{\frac{D_M}{2}}, && C_2=2^{D_M-1}D_M^{D_M},
\end{eqnarray*}
valid when $D_{M}+n \geq \e+2.$ Also, when we consider $M(r)=r^p,$
no restrictions other that $p>2$ are required, which follows from
a straightforward calculation which uses integration by parts and
H\"{o}lder's inequality only. See also Corollary \ref{p-space}
below.

 For $\rho>1$ let
\[
\beta(\rho)=\sup_{w>0} \left\{ \left(\frac{1}{2}w+\sqrt{\frac{1}{4}w^{2}+1}\right)^{D_{M}}-\rho w^{D_{M}}\right\}
\]
and
\[
\gamma(\rho)=\sup_{w>0} \left\{ \left(\frac{1}{2}+\sqrt{\frac{1}{4}+w^{2}}\right)^{D_{M}}-\rho w^{D_{M}}\right\}.
\]
Obviously, $\beta(\rho)$ and $\gamma(\rho)$ are finite but they grow to infinity as $\rho \to 1^{+}$.
By simple considerations involving homogeneity we prove that (\ref{term2}) implies
\[
\K \leq \beta(\rho) \cdot (D_{M}+n-2)^{D_{M}/2} \L + \rho \cdot D_{M}^{D_{M}} \M
\]
and
\[
\K \leq \rho\cdot (D_{M}+n-2)^{D_{M}/2} \L+\gamma(\rho)\cdot D_{M}^{D_{M}} \M.
\]
Thus, under assumptions of Proposition \ref{alternatywa}, we may obtain  (\ref{liniowe}) with any $C_{1}$ greater than
$\max\left( (D_{M}+n-2)^{D_{M}/2}, (D_{M}/d_{M})^{D_{M}/(D_{M}-2)}\right)$  -- note that this quantity does not exceed
$\left( \max(D_{M}+n-2, \e) \right)^{D_{M}/2}$. However, this comes at the expense of $C_{2}$ getting large. Similarly, under
the same assumptions, we may prove (\ref{liniowe}) with any $C_{2}>D_{M}^{D_{M}}$, at the expense of $C_{1}$ getting
large.

If $M$ is a power function, $M(r)=r^{p}$ for some $p>2$, then $D_{M}=d_{M}=p$, and we obtain the following corollary
to Proposition \ref{alternatywa}.

\begin{cor} \label{p-space}
Assume that $n \geq 1$ and $p>2$. Let $\ud \mu_{n}=r^{n-1}\e^{-r^{2}/2}\,\ud r$. For an a.e. differentiable function
$u: [0,\infty) \rightarrow \R$ let
\[
\K=\int_{0}^{\infty} (r|u(r)|)^{p}\,\ud \mu_{n}(r),\;\;
\L=\int_{0}^{\infty} |u(r)|^{p}\,\ud \mu_{n}(r),\;\;
\M=\int_{0}^{\infty} |u'(r)|^{p}\,\ud \mu_{n}(r).
\]
Then for every $C_{2}>p^{p}$ there exists some $C_{1}=C_{1}(n,p,C_{2}) < \infty$,
and  for every $C_{1}>(n+p-2)^{p/2}$ there exists some $C_{2}=C_{2}(n,p,C_{1}) < \infty$,
such that for every continuous and piecewise $\C^{1}$ function $u: [0,\infty) \rightarrow \R$
there is
\[
\K \leq C_{1}\L+C_{2}\M.
\]
\end{cor}

\begin{proof} Preceding considerations.
\end{proof}

\begin{rem}{\rm It is known (see \cite{bct}) that for $p=2$ one
has
\[\frac{1}{4} \K \leq \L+ \frac{n}{2}\, M ,\]
(i.e. $C_1=2n,$ $C_2=4$)
and that the constant $\frac{1}{4}$ cannot be improved.
 Additionally, if $C_{2}=4$ then (\ref{liniowe})
holds true with $C_{1}=2n$ but it fails for $C_{1} < 2n.$
In this case ($p=2$), our method permits to lower $C_1$
as close to $n$  as we wish, again at the expense of getting $C_2$
large. Getting $C_1=n$ is not possible.}\end{rem}

 The constants $C_{1}$ and $C_{2}$ in Corollary
\ref{p-space} (and thus also the bounds of Proposition
\ref{alternatywa}) are quite good. Indeed, assume that
(\ref{liniowe}) holds with some constants $C_{1}$ and $C_{2}$. Let
$\alpha \in [0,1)$. A straightforward calculation yields that for
$u(r)=\exp\left(\frac{\alpha r^{2}}{2p}\right)$ there is
\begin{eqnarray*}
\K=\int_0^\infty (ru_\alpha(r))^p r^{n-1} \e^{r^{2}/2}\,\ud r&=&
\int_0^\infty r^{n+p-1}\e^{-\frac{(1-\alpha)r^2}{2}}\,\ud r=\\
(1-\alpha)^{-(n+p)/2}\int_0^\infty
\rho^{n+p-1}\e^{-\rho^{2}/2}\,\ud\rho
&=&(1-\alpha)^{-(n+p)/2}
2^{(n+p-2)/2}\,\Gamma(\textstyle\frac{n+p}{2}),
\end{eqnarray*}
\begin{eqnarray*}
\L=\int_0^\infty (u_\alpha(r))^p r^{n-1} \e^{-r^{2}/2}\,\ud r&=&
(1-\alpha)^{-n/2}
2^{(n-2)/2}\,\Gamma(\textstyle\frac{n}{2}),\\
\M=\int_0^\infty (u'_\alpha(r))^p r^{n-1} \e^{-\frac{r^2}{2}}\,\ud r&=&
(\alpha/p)^p (1-\alpha)^{-(n+p)/2}
2^{(n+p-2)/2}\,\Gamma(\textstyle\frac{n+p}{2}).
\end{eqnarray*}
so that
\begin{equation} \label{alfa}
(1-\alpha)^{-p/2}
2^{p/2}\Gamma(\textstyle\frac{n+p}{2})\leq C_1
\Gamma(\textstyle\frac{n}{2}) +C_2 (\alpha/p)^p
(1-\alpha)^{-p/2}
2^{p/2}\Gamma(\textstyle\frac{n+p}{2}),
\end{equation}
Were $C_{2} \leq p^{p}$, (\ref{alfa}) would imply that $C_{1} \geq \frac{2^{p/2}\Gamma((n+p)/2)}{\Gamma(n/2)}
\frac{1-\alpha^{p}}{(1-\alpha)^{p/2}} \to \infty$ as $\alpha \to 1^{-}$. The obtained contradiction proves that in general (\ref{liniowe})
cannot hold with $C_{2} \leq p^{p}$.

Moreover, by taking $\alpha=0$ in (\ref{alfa}) we immediately see that in general (\ref{liniowe}) cannot hold with
$C_{1}<2^{p/2}\Gamma( (n+p)/2)/\Gamma(n/2)$. Note that by Stirling's formula we have
\[
\lim_{n \to \infty} \frac{2^{p/2}\Gamma( (n+p)/2)}{(n+p-2)^{p/2}\Gamma(n/2)}=1,
\]
so the assumption that $C_{1}>(n+p-2)^{p/2}$ in Corollary \ref{p-space} is also (asymptotically) quite tight.

Proposition \ref{alternatywa} provides reasonable bounds but its assumptions are a bit restrictive in that they require the function
$r \mapsto r^{-2}M(r)$ to be non-decreasing. However, we may also prove (\ref{liniowe})--type inequality if we replace (\ref{gwgw}) by convexity.
This time we do not push for the best possible constants.

\begin{prop} \label{wklep1}
Let $M:[0,\infty) \rightarrow [0,\infty)$ be an increasing and convex
function with $M(0)=0$. Assume that for some $D_M>0,$
$M(\alpha x) \leq \alpha^{D_{M}}M(x)$ for any $\alpha \geq 1$, $x
\geq 0$ (doubling). For $n \geq 1$ let $\ud
\mu_{n}(r)=r^{n-1}\e^{-r^{2}/2}\,\ud r$. Then for any continuous,
piecewise $\C^{1}$ function $u: [0,\infty) \rightarrow \R$
we have
\begin{equation}\label{ww}
\int_{0}^{\infty} M(r|u(r)|)\,\ud \mu_{n}(r) \leq
C_{1} \cdot \int_{0}^{\infty} M(|u(r)|)\,\ud \mu_{n}(r)+
C_{2} \cdot \int_{0}^{\infty} M(|u'(r)|)\,\ud \mu_{n}(r),
\end{equation}
with $C_{1}$ and $C_{2}$ depending only on $n$ and $D_{M}$.
\end{prop}

\noindent Since $M$ is convex and increasing there must be $D_{M}
\geq 1$. Observe that when $M$ is an $N-$ function satisfying
the $\Delta_2-$condition, then the assumptions of Proposition
\ref{wklep1} are satisfied.

 We need a simple lemma.

\begin{lem} \label{wklep2}
For every $\eps \in (0,1]$ and every $a, b \geq 0$ we have
\[
M(a)b \leq \eps M(a)+\eps^{-D_{M}}M(ab).
\]
\end{lem}

\begin{proof}
If $b \geq 1$ then
\[
M(a)=M(b^{-1} \cdot ab+(1-b^{-1})\cdot 0) \leq b^{-1}M(ab)+(1-b^{-1})M(0)=b^{-1}M(ab)
\]
and the inequality obviously holds. For $b \in [0,\eps]$ the inequality is trival. For $b \in (\eps,1)$ we have
\[
M(a)b \leq M(a)=M(b^{-1} \cdot ab) \leq
(1/b)^{D_{M}}M(ab) \leq \eps^{-D_{M}}M(ab).
\]
\end{proof}

\noindent {\em Proof of Proposition \ref{wklep1}.} Again, we first
assume additionally that $u$ is compactly supported. Let
\[
\K=\int_{0}^{\infty} M(r|u(r)|)\, \ud \mu_{n}(r),
\L=\int_{0}^{\infty} M(|u(r)|)\,\ud \mu_{n}(r),
\M=\int_{0}^{\infty} M(|u'(r)|)\,\ud \mu_{n}(r).
\]
For any $\kappa \geq 1$ we have $\L \geq \e^{-(\kappa+1)^{2}/2} \cdot \int_{\kappa}^{\kappa+1} M(|u(r)|)\,\ud r,$
so that there exists $\tilde{r} \in [\kappa, \kappa+1]$ such that $M(|u(\tilde{r})|) \leq \e^{2\kappa^{2}}\L$. We also have
\[
M\left( \int_{\kappa}^{\kappa+1} |u'(r)|\,\ud r\right) \leq \int_{\kappa}^{\kappa+1} M(|u'(r)|)\,\ud r \leq
\]
\[
\e^{(\kappa+1)^{2}/2} \int_{\kappa}^{\kappa+1} M(|u'(r)|)r^{n-1}\e^{-r^{2}/2}\,\ud r \leq \e^{2\kappa^{2}}\M.
\]
Hence
\[
M(|u(\kappa)|) \leq M\left(|u(\tilde{r})|+\int_{\kappa}^{\kappa+1} |u'(r)|\,\ud r\right) \leq
M\left(2\max\left(|u(\tilde{r})|,\int_{\kappa}^{\kappa+1} |u'(r)|\,\ud r\right)\right) \leq
\]
\[
2^{D_{M}}\max\left(M(|u(\tilde{r})|),M\left(\int_{\kappa}^{\kappa+1} |u'(r)|\,\ud r\right)\right) \leq
2^{D_{M}}\e^{2\kappa^{2}}\max(\L,\M) \leq 2^{D_{M}}\e^{2\kappa^{2}}(\L+\M).
\]
We have
\[
\K=\int_{0}^{\kappa} M(r|u(r)|)r^{n-1}\e^{-r^{2}/2}\,\ud r+\int_{\kappa}^{\infty} M(r|u(r)|)r^{n-1}\e^{-r^{2}/2}\,\ud r \leq
\]
\[
\kappa^{D_{M}}\L-\int_{\kappa}^{\infty} M(r|u(r)|)r^{n-2}\frac{\ud}{\ud r}(\e^{-r^{2}/2})\,\ud r \leq
\]
\[
\kappa^{D_{M}}\L+M(\kappa|u(\kappa)|)\kappa^{n-2}\e^{-\kappa^{2}/2}+(n-2)\int_{\kappa}^{\infty} M(r|u(r)|)r^{n-3}\e^{-r^{2}/2}\,\ud r+
\]
\[
D_{M}\int_{\kappa}^{\infty} \frac{M(r|u(r)|)}{r|u(r)|}(|u(r)|+r|u'(r)|)r^{n-2}\e^{-r^{2}/2}\,\ud r \leq
\]
\[
\kappa^{D_{M}}\L+2^{D_{M}}\e^{2\kappa^{2}}\kappa^{D_{M}+n-2}(\L+\M)+
\kappa^{-2}(D_{M}+n-2)\K+
\]
\[
D_{M}\int_{(\kappa, \infty) \cap \{ u \neq 0\}} M(r|u(r)|)\left|\frac{u'(r)}{ru(r)}\right|r^{n-1}\e^{-r^{2}/2}\,\ud r
\stackrel{\hbox{{\small{Lem. \ref{wklep2}}}}}{\leq}
\]
\[
\kappa^{D_{M}}\L+2^{D_{M}}\e^{2\kappa^{2}}\kappa^{D_{M}+n-2}(\L+\M)+\kappa^{-2}(D_{M}+n)\K+\eps D_{M}\K+\eps^{-D_{M}}D_{M}\M.
\]
By setting $\eps=(4D_{M})^{-1}$ and $\kappa=2(D_{M}+n)^{1/2}$, upon obvious cancellations we obtain the asserted estimate.
Finally, we may remove the compact support assumption in the same way as in the proof of Proposition \ref{alternatywa}.
\qed

When $M$ is an $N-$function satisfying the $\Delta_2-$condition, then using standard Orlicz-space methods one
can obtain the Hardy inequality for norms.

\begin{cor}\label{norms}
Suppose that $M:\zi\to\zi$ is an $N-$function satisfying the $\Delta_2-$condition. Then the assertion of Proposition
\ref{wklep1} holds true, and, moreover, there exists a constant $C>0$ such that for any continuous,
piecewise $\C^{1}$ function \linebreak
$u: [0,\infty) \rightarrow \R$
\begin{equation}\label{www}
 \|x u(x)\|_{L^{M}(\zi,\mu_n)} \leq C\left( \|u(x)\|_{L^{M}(\zi,\mu_n)}+\|u'(x)\|_{L^{M}(\zi,\mu_n)}\right).
 \end{equation}
 \end{cor}

\begin{proof} We only need to prove (\ref{www}).
 For
short, write $\|u\|_M$ instead of $\|u\|_{L^M(\zi,\mu_n)}.$ For
a given nonconstant function $u,$ consider
$\widetilde{u}=\frac{u}{\|u\|_M+\| u'\|_M},$ and write
(\ref{ww}) for function $\widetilde u:$
\begin{eqnarray*}
\int_0^\infty M(|x\widetilde u(x)|)\,\ud\mu_n(x) &\leq  & C_{1}
\int_0^\infty M\left(|\widetilde {u}(x)|\right)\,\ud\mu_n +C_2
\int_0^\infty M\left(|\widetilde{u}'(x)|\right)\,\ud\mu_n
\\
&\leq&  C_{1} \int_0^\infty
M\left(\frac{|u|}{\|u\|_M}\right)\,\ud\mu_n +C_2 \int_0^\infty
M\left(\frac{| u'|}{\| u'\|_M}\right)\,\ud\mu_n\\[2mm]
& =&  C_1 +C_2
\end{eqnarray*}
(we have used (\ref{norm1})). It follows that $\|x\widetilde
u(x)\|_M\leq C_1+C_2+1,$ and since $\|\cdot\|_M$ is a
norm, (\ref{www}) follows.
\end{proof}

\subsection{The $n$-dimensional case}\label{endim} Using
the one-dimensional inequality as a tool, now we derive the Hardy
inequality for the $n-$dimensional Gaussian measure. We start with the
statement under general assumptions on the function $M$ involved,
which however does not give a good control on the resulting
constants.

\begin{theo}\label{HN}
Let $M:[0,\infty) \rightarrow \infty$ be an increasing and convex
function with $M(0)=0$. Assume that $M(a x) \leq a^{D_{M}}M(x)$
for any $a \geq 1$, $x \geq 0$ (doubling). Let $n\geq 1,$ and let
${\rm d}\gamma_n(x) = {\rm e}^{-|x|^2/2}.$
 Then for $u\in C^1_0(\rn)$ we have:
\begin{equation}\label{hn1}
\int_\rn M(|x|\, |u(x)|)\,\ud\gamma_n(x)           \leq
C_1 \int_\rn M(|u(x)|)\,\ud\gamma_n(x)+C_2 \int_\rn
M(|\nabla u|)\,\ud\gamma_n(x),
\end{equation}
where $ C_1, C_2$
 are  constants from
Proposition \ref{wklep1}. When $M$ is an $N-$function, then also
\begin{equation}\label{hn11}
\| u(x)\,x\|_{L^M(\rn,\gamma_n)}\leq
C\left(\|u(x)\|_{L^M(\rn,\gamma_n)}+\|\nabla
u\|_{L^M(\rn,\gamma_n)}\right).
\end{equation}
\end{theo}

\begin{proof}
We start with the proof of (\ref{hn1}).
We can write,
in spherical coordinates,
\begin{eqnarray*}
\int_\rn M(|x|\,|u(x)|)\,\ud\gamma_n(x)&= &
\omega_{n}\int_{S^{n-1}}\int_0^\infty M(|r
u(r,y)|)r^{n-1}\e^{-\frac{r^2}{2}}\,\ud r\,\ud\sigma_{n-1}(y)  ,
\end{eqnarray*}
where $\sigma_{n-1}$ denotes the normalized surface measure on
$S^{n-1}\!\subset\!\rn,$ $\omega_{n}$ is the standard
$(n-1)$-dimensional surface measure of $S^{n-1}$, and
$u(r,y)=u(ry)$ (for $n=1$, we do not have to do anything). For $y$
fixed, the function $v(r)=u(r,y)$ is a $\C^1-$function,  and so we
can apply Proposition \ref{wklep1} for $v(r)= u(r,y).$ We obtain
(for given $y$):
\[
\int_0^\infty M(r|v(r)|)r^{n-1}\e^{-r^{2}/2}\,\ud r \leq
\]
\[
C_1 \int_0^\infty M(|v(r)|)r^{n-1}\e^{-r^{2}/2}\,\ud r
+C_2 \int_0^\infty M(|v'(r)|)r^{n-1}\e^{-r^{2}/2}\,\ud r,
\]
and the constants $C_1$, $C_2$ {\em do not} depend on
$y.$ Note also that $v'(r)=\frac{\partial u}{\partial r}(r,y),$
and so $|v'(r)|\leq |\nabla u(x)|.$ Switching back to Euclidean
coordinates we obtain:
\[\int_\rn M(|x|\,|u(x)|)\,\ud\gamma_n(x)\leq C_1 \int_\rn M(|u(x)|)\,\ud\gamma_n(x)+C_2
\int_\rn M(|\nabla u(x)|)\,\ud\gamma_n(x).\]
 (\ref{hn1}) is proven.

To get (\ref{hn11}), we use a standard Orlicz-space argument similar to that in
the proof of Corollary \ref{norms}: we apply (\ref{hn1}) to the function $\widetilde{u}=\frac{u}{\|u\|_M+\|\nabla u\|_M},$
and then we proceed as before.
\end{proof}

Under more restrictive assumptions on $M,$ we can use Proposition
\ref{alternatywa} instead of Proposition \ref{wklep1}, so that the
constants are better controlled. For example, when  \linebreak $D_M+n-2\geq
{\rm e},$ we get:

\begin{theo}\label{hn2 }
Suppose that $M$ satisfies condition {\bf (M)} with $d_M\geq 2$
and $D_M>\max \{2, {\rm e}+2-n\}.$  Let $u\in C^1_0(\rn).$
Denoting
\begin{eqnarray*}
&& \K^{(n)}=\int_\rn \! M(|x|\,|u(x) |)\,\ud \gamma_{n}(x),\;\;\;
\L^{(n)}=\int_\rn\! M(|u(x)|)\,\ud \gamma_{n}(x)\\
&& \M^{(n)}=\int_\rn\! M(|\nabla u(x)|)\,\ud \gamma_{n}(x),
\end{eqnarray*}
one gets
\begin{equation}\label{wwww}
\K^{(n)} \leq \left(
\frac{1}{2}D_{M}\left(\M^{(n)}\right)^{1/D_{M}}+\sqrt{\frac{1}{4}D_{M}^{2}\left(\M^{(n)}\right)^{2/D_{M}}+
(D_{M}+n-2)\left(\L^{(n)}\right)^{2/D_{M}}}
\right)^{D_{M}}.
\end{equation}

\end{theo}

The proof is identical with that of Theorem \ref{HN}.

\section{The Landau-Kolmogorov inequality for the Gaussian\\
measure}

The Hardy inequalities from Section \ref{endim} can be used for
deriving Landau-Kolmogorov inequalities for Gaussian measures in
$\rn.$

To this end, we will use the following theorem (Theorem 3.3 of
\cite{AKKPPbullpan}), applied with $P=Q=M.$

\begin{theo}[\cite{AKKPPbullpan}]\label{theoB} Let
$\Omega\subset\rn$be an open domain.  Suppose that $M$ is a
differentiable $N-$function satisfying the $\Delta_2-$condition
and such that $M(r)/r^2$ is non-decreasing. Let ${\rm d}\mu(x)=
{\rm e}^{-\varphi(x)}{\rm d}x$ be a Radon measure on $\Omega$ such
that $\varphi\in W_{loc}^{1,\infty}(\Omega)$. If for every $u\in
C^\infty_0(\Omega)$ the following
 Hardy-type
inequality holds true:
\begin{equation}\label{assumB1}
\int_\Omega M(|\nabla\varphi|\,|u|)\,\ud\mu\leq K_1\int_\Omega
M(|\nabla u|)\,\ud\mu +K_2\int_\Omega M(|u|)\,\ud\mu,
\end{equation}
then we have:
\begin{description}
\item[1)] there exist positive constants $C_1,C_2$ such that
for  any $\theta\in (0,1]$ and  any $u\in C_0^\infty(\Omega)$
 \begin{equation}\label{statB1} \int_\Omega M(|\nabla u|)\,\ud\mu\leq
C_1\int_\Omega M(\theta |\nabla^{(2)}u|)\,\ud\mu +C_2 \int_\Omega
M(|u|/\theta)\,\ud\mu;
\end{equation}
\item[2)] there exist positive constants $\tilde{C}_1,\tilde{C}_2$ such that
for any $u\in C_0^\infty(\Omega)$
 \begin{equation}\label{statB2} \|\nabla
u\|_{L^M(\Omega,\mu)}\leq\tilde C_1
\sqrt{\|\nabla^{(2)}u\|_{L^M(\Omega,\mu)} \|u\|_{L^M(\Omega,\mu)}}
+\tilde C_2 \|u\|_{L^M(\Omega,\mu)}.
\end{equation}
\end{description}
\end{theo}

We apply this theorem to $\Omega=\rn $ and  ${\rm d}\mu(x)= {\rm
e}^{-|x|^2/2}{\rm d}x.$ In this case $|\nabla \varphi (x)|= |x|,$
and the validity of (\ref{assumB1}) is assured by Proposition
\ref{wklep1} (or Proposition \ref{alternatywa}, provided we assume
{\bf (M)}). Choosing $\theta=1$  we obtain the following:

\begin{cor}\label{forgausssian}
Suppose $M$ is a differentiable $N-$function satisfying the
$\Delta_2-$condition and such that $M(r)/r^2$ is non-decreasing.
Let $\ud \gamma_n(x)= {\rm e}^{-|x|^2/2} \ud x.$ Then there exist
positive constants $C_1,C_2$ such that for any $u\in
C_0^\infty(\rn)$ one has
\begin{equation}\label{statB1gauss} \int_\rn M(|\nabla u|)\,\ud\gamma_n\leq
C_1\int_\rn M( |\nabla^{(2)}u|)\,\ud\gamma_n\ +C_2 \int_\rn
M(|u|)\,\ud\gamma_n;
\end{equation}
and  positive constants $\tilde{C}_1,\tilde{C}_2$ such that for
any $u\in C_0^\infty(\Omega)$
 \begin{equation}\label{statB2gauss} \|\nabla
u\|_{L^M(\rn,\gamma_n)}\leq\tilde C_1
\sqrt{\|\nabla^{(2)}u\|_{L^M(\rn,\gamma_n)}
\|u\|_{L^M(\rn,\gamma_n)}} +\tilde C_2 \|u\|_{L^M(\rn,\gamma_n)}.
\end{equation}
\end{cor}

By usual density arguments, smoothness conditions on $u$ can be
relaxed.


\begin{thebibliography}{99}



\bibitem{hama} Bang, H.H.,  Thu, M.T.,  { A Landau-Kolmogorov
inequality for Orlicz spaces}, {\em J. Inequal. Appl.} {\bf 7}
(2002), no. 5, 663--672.

\bibitem{bla} di Blassio, G., Linear Elliptic Equations and Gauss
measure, {\em J. Ineq. in. Pure Appl. Math.} {\bf 4}(5), article 106
(2003), electronic.


\bibitem{bfp} di Blassio, G., Feo, F., Posteraro, M.R.,
Existence results for nonlinear elliptic equations related to
Gauss measure in a limit case.  {\em Commun. Pure Appl. Anal.}  7
(2008),  no. 6, 1497--1506.

\bibitem{bct}
Brandolini, B.,  Chiaccio, F., Trombetti, C., {Hardy type
inequalities and Gaussian measure}, {\em Comm. Pure and Appl.
Anal.} {\bf 6}(2) (2007), 411--428.



\bibitem {Heb-Zeg}
 Hebisch, W.,  Zegarlinski, B., { Coercive inequalities on metric measure
spaces.},  {\em J. Funct. Anal.}, 258, 814--851 (2010).


\bibitem{AKKPPBBMS} Ka\l amajska A., Pietruska-Pa\l uba, K.,
Gagliardo-Nirenberg inequalities in weighted Orlicz spaces
equipped with a nonnecessarily doubling measure, {\em Bull. Belg.
Math. Soc. Simon Stevin} {\bf 15}, 1--19 (2008).



\bibitem{AKKPPbullpan} Ka\l amajska A., Pietruska-Pa\l uba, K.,
 On a variant of the Gagliardo-Nirenberg Inequality
 Deduced from the Hardy Inequality,
 {\em Bull. Pol. Acad. Sci. Math.} {\bf 59}, 133--149 (2011).



\bibitem{led} Ledoux, M., {\em The concentration of measure phenomenon}.
Mathematical Surveys and Monographs 89, American Mathematical
Society (2001).

\bibitem{ma}   Maz'ya, V.\,G., {\em Sobolev Spaces}, Springer--Verlag
1985.


\bibitem{rao-ren} Rao, M.M.,    Ren, Z.D., {\em Theory of Orlicz
spaces}, M. Dekker, Inc. New York, 1991.

\bigskip

\noindent
{\small K.\,Oleszkiewicz:\\
Institute of Mathematics, University of Warsaw, ul.~Banacha~2, 02-097 Warsaw, Poland;\\
Institute of Mathematics, Polish Academy of Sciences, ul.~\'Sniadeckich 8, P.O. Box 21,\\
00-956 Warsaw, Poland; e-mail: koles@mimuw.edu.pl}

\bigskip
\noindent
{\small K.\,Pietruska-Pa{\l}uba:\\
Institute of Mathematics, University of Warsaw, ul.~Banacha~2, 02-097 Warsaw, Poland;
e-mail: kpp@mimuw.edu.pl}

\end{thebibliography}
\end{document}